\documentclass[11pt,english,reqno]{amsart}
\usepackage[T1]{fontenc}
\usepackage[latin1]{inputenc}
\usepackage{amsbsy,amsthm,amsmath,mathrsfs}
\usepackage{amstext}
\usepackage{amssymb}

\theoremstyle{plain} 
\newtheorem{theorem}{Theorem}[section]
\newtheorem{thm}[theorem]{Theorem}
\newtheorem{lemma}[theorem]{Lemma}

\newtheorem{proposition}[theorem]{Proposition}  

\newtheorem{corollary}[theorem]{Corollary} 
 
\theoremstyle{definition}
\newtheorem{example}[theorem]{Example}

\newtheorem{remark}[theorem]{Remark}

\newtheorem{definition}[theorem]{Definition}

\newcommand{\RR}{\mathrel{\mathscr{R}}}
\newcommand{\LL}{\mathrel{\mathscr{L}}}
\newcommand{\HH}{\mathrel{\mathscr{H}}}
\newcommand{\DD}{\mathrel{\mathscr{D}}}
\newcommand{\JJ}{\mathrel{\mathscr{J}}}

\newcommand{\up}[1]{\textup{#1}}

\makeatletter

\providecommand{\tabularnewline}{\\}

\makeatother

\usepackage{babel}
\begin{document}

\title{\textbf{Algebras defined by equations }}

\author{Peter M. Higgins}
\address{Department of Mathematical Sciences, University of Essex, UK}
\email{peteh@essex.ac.uk}
\author{Marcel Jackson}
\address{Department of Mathematics and Statistics, La Trobe University, VIC 3086, Australia}
\email{m.g.jackson@latrobe.edu.au}
\thanks{The second author was partially supported by ARC Future Fellowship FT120100666}
\maketitle
\begin{abstract}
We show that an elementary class of algebras is closed under the taking of homomorphic
images and direct products if and only if the class consists of all
algebras that satisfy a set of (generally simultaneous) equations.
For classes of regular semigroups this allows an interpretation
of a universal algebraic nature that is formulated entirely in terms
of the associative binary operation of the semigroup, which serves
as an alternative to the approach via so called e-varieties. In particular
we prove that classes of Inverse semigroups, Orthodox semigroups,
and $E$-solid semigroups are equational in our sense. We also examine which equations are valid in every semigroup. 
\end{abstract}
\section{Introduction}\label{sec:intro}

Groups may be characterized in terms of their binary operation alone
as they form the class of semigroups that are both left and right
simple, which is to say that a semigroup $S$ is a group if and only
if $aS=Sa=S$ for all $a\in S$. Since this introduces the theme of
the paper,  let us observe that the given pair of conditions on $S$
may be expressed by saying that the equations $ax=b$ and $ya=b$
are always solvable in $S$, meaning that the class $\mathcal{G}$ of
all groups is defined within the class of semigroups by the equations:
\begin{equation}
\text{\ensuremath{\mathcal{G}}\ensuremath{:\,}}(\forall a,b\in S)\,(\exists x,y\in S)\colon (ax=b)\wedge(ya=b).\label{eqn:group}
\end{equation}
A second observation is that $\mathcal{G}$ is a class of semigroups
closed under the operations $\mathsf{H}$ and $\mathsf{P}$, which are respectively
the taking of homomorphic images, and the taking of direct products,
but $\mathcal{G}$ is not closed under the taking of subsemigroups, so
that $\mathcal{G}$ does not constitute a semigroup variety.  Many fundamental
semigroup classes are $\{\mathsf{H},\mathsf{P}\}$-closed classes in this way and we may easily
identify natural equational bases, as we show in Section \ref{sec:examples}.  
In general we will use the phrase \emph{equation system} in preference to simply \emph{equation}, to allow for the fact that they are typically systems of simultaneous equations (grouped by conjunction) and that we allow arbitrary quantification.  This also avoids confusion with the common use of  ``equation'' synonymously with ``identity'' in the context of varieties. Nevertheless, we allow \emph{equational basis} to refer to any family of equation systems that characterise a class.

In Section \ref{sec:EHP} we give a fundamental model theoretic theorem (originally noted, though not explicitly proved, by Keisler~\cite{kei}) underlying
this idea, which is that an elementary class $\mathcal{C}$ of algebras is $\{\mathsf{H},\mathsf{P}\}$-closed
if and only if $\mathcal{C}$ consists of all algebras for which there
exist solutions to a certain set of equation systems.  The reverse direction
is clear but the forward implication is a consequence of Lyndon's
Positivity Theorem (see \cite{lyn} or Corollary~8.3.5 of \cite{hod}).  In
Section 4 we find equational bases for the $\{\mathsf{H},\mathsf{P}\}$-classes of Inverse
semigroups, Orthodox semigroups, and $E$-solid semigroups (semigroups
whose idempotent generated part is a union of groups). Section 5 compares our approach to that of 
e-varieties when applied to classes of regular semigroups.  The final section is on equation systems 
that are universally solvable in any semigroup. 

General background on semigroup theory will be assumed. We direct
the reader to the books \cite{clipre,hig,how,lal} and our textual source for universal
algebra is \cite{bursan}. Standard location theorems for Green's relations
and properties of regularity will often be used without further comment.
The symbol $S$ stands for a semigroup and we denote the set of idempotents
of $S$ by $E(S)$ or sometimes simply by $E$. We write $V(A)$ to
denote the set of inverses of members of the subset $A$ of $S$.
One fact drawn upon in Section \ref{sec:evar} is that in a regular semigroup $S$,
$V(E^{n})=E^{n+1}$, from which it follows that the \emph{core} of $S$, which is
the  idempotent generated subsemigroup $\langle E\rangle$ of $S$, is itself regular (see Fitzgerald \cite{fitz}).

\section{Examples of equational bases for $\{\mathsf{H},\mathsf{P}\}$-classes of semigroups}\label{sec:examples}
In the following, $S$ always denotes a semigroup, and unless otherwise stated, quantification is over elements of $S$.
\begin{example}
\begin{enumerate}
\item[(i)] $\mathcal{R}eg$, the class of all \emph{regular semigroups} may
be defined by the single equation $a=axa$, which is to say 
\begin{equation}
\text{\ensuremath{\mathcal{R}eg}:}\ (\forall a)\,(\exists x)\colon  axa=a.
\end{equation}
\end{enumerate}

The classes of Completely regular semigroups, Semilattices of groups, and of Completely simple semigroups may each be defined within the class of regular semigroups by one additional equation.

\begin{enumerate}

\item[(ii)] $\mathcal{CR}$, the class of all \emph{completely regular semigroup}s
(unions of groups) has an equational basis in our sense given by:
\begin{equation}
\text{\ensuremath{\mathcal{CR}:}}\ (\forall a)\,(\exists x)\colon (a=axa)\wedge(ax=xa)\label{eq:CR}
\end{equation}
for if $S\in\mathcal{CR}$ then for any $a\in S$ we take $x$ as the
group inverse of $a$ in order to satisfy the equation system. Conversely, given that $S$ satisfies
this equation system we have that $y=xax\in V(a)$ and $ay = axax = ax = xa = xaxa = ya$.  But then we have $a\HH y$ as $a=a^{2}y=ya^{2}$
and $y=y^{2}a=ay^{2},$ so that $H_{a}$ is a group and therefore $S$ is a union of groups.

\item[(iii)] The class $\mathcal{SL}$ of all \emph{semilattices
of groups} may be defined by: 

\begin{equation}
\mathcal{SL}\colon (\forall a,b)\,(\exists x,y)\colon (a=axa)\wedge (ab=bya).
\end{equation}
To see this, given that $S\in\mathcal{SL}$ then the first equation is
satisfied by regularity. Now $H_{ab}=H_{ba}=H$, a group with identity
$e$ say. Hence $be,ea\in H$ so we put $y=e(be)^{-1}ab(ea)^{-1}e$
(where inversion is in the group $H$) and then 
\[
bya=be(be)^{-1}ab(ea)^{-1}ea=e(ab)e=ab.
\]
Conversely if $S$ satisfies the given equations then $S$ is certainly
regular.  Take $a\in S,\,e\in E(S)$. Then there exists $y\in S$ such
that $ae=eya$, whence $eae=e^{2}ya=eya=ae$. Similarly there exists
$z\in S$ such that $ea=aze$, whence $eae=aze^{2}=aze=ea$. Therefore
$ae=eae=ea$, which shows that idempotents are central and we conclude
that $S$ is a semilattice of groups.  

We note also that Theorems 5.1 and  5.2 of~\cite{nag}  show that the second equation of (4) taken over $S^1$  characterises  semigroups in which $\mathscr{H}$  is a congruence such that $S/\mathscr{H}$  is commutative. 

\item[(iv)] The class $\mathcal{CS}$ of \emph{completely simple semigroups}
is defined by: 
\begin{equation}
\mathcal{CS}\colon (\forall a,b)\,(\exists x,y)\colon (a=axa)\wedge(a=abay).\label{eq:CS}
\end{equation}
For if $S$ is a completely simple semigroup then $S$ is certainly regular and since
for any $a,b\in S$ we have $a\HH aba$, it follows that there are always solutions
to the second equality in \eqref{eq:CS} as well. 
Conversely, given that $S$ satisfes
the equation system we see that  $S$ is regular while
the second equality in~\eqref{eq:CS} implies that $a\leq_{\JJ}b$ is true
for all $a,b\in S$ and so $S$ is simple.  Suppose that $e\leq f$  holds in the natural 
partial order of the idempotents of $S$ so that $e = ef = fe = fef$.  Taking $a = f$  and $b = e$ in 
the second equation gives us  $f = fefy = ey$, and so $e = ef = e^2y = ey = f$.  
It follows that each idempotent $f \in E(S)$  is primitive and therefore $S$ is indeed 
completely simple. 
\end{enumerate}
\end{example}
We adopt the convention in our equations that letters taken from the front of the alphabet, $a,b,c$
are \emph{parameters}, which means they are quantified by a $\forall$
symbol, while $x,y,z$ denote \emph{variables}, meaning that they
are quantified by the symbol $\exists$. 

 We may sometimes abbreviate certain collections of equalities by expressions
that are shorter and the meaning of which is clearer. For example, $x \in V(a)$ 
is equivalent to $(a = axa)\wedge (x = xax)$. However, if
the equalities required are simultaneous, meaning that they contain
common variables, these abbreviations may not suffice and the equations
may need to be listed explicitly to convey the required duplication
of variables between equations. However the equations $((\exists x)\,(\forall a):\,ax=xa=a)$
may be shortened to $x=1$ and similarly $((\exists x)\,(\forall a):\,ax=xa=x)$
can be written as $x=0$. When dealing with long strings it is sometimes
convenient to write the equation $w=w^{2}$ as $w\in E$, although
this is an abuse of notation as $w$ is a word in a free semigroup
pre-image of $S$ while $E=E(S)$. Another example we shall make use of is to write $w \in G$  to indicate that $w$ belongs to a subgroup of~$S$.  The pair of equations $(\exists x):(x\in V(w)) \wedge (wx = xw)$ is equivalent to the pair of conditions that  $x$  is an inverse of $w$  and $H_w = H_x$  is a group $\HH$-class.  In this way the class $\mathcal {CR}$ of completely regular semigroups may now simply be expressed by $(\forall a) (a \in G)$. 

\begin{example}
\begin{enumerate}
\item[(i)] Let $\mathcal{IG}$ be the class of all semigroups $S$ for which
each element has an inverse that lies in some subgroup of $S$: $\mathcal{IG}$ may be captured as
the conjunction of equational properties as described above:
\[
\mathcal{IG}:\,(\forall a)\,(\exists x,y)\colon (x\in V(a))\wedge(x\HH y)\wedge(y\in E(S)).
\]
We note that expressions such as $x \HH  y$  may be expressed by equations but that in general these are equations over $S^{1}$ rather than $S$.  In the sequel it will be seen that, on some occasions, this is a significant distinction although that is not so in the presence of regularity (as we have in this example) and so the relevant equations can indeed be taken over $S$. 

\item[(ii)] By the class $\mathcal{C}r$ of \emph{cryptogroups} is meant those
semigroups $S$ that are completely regular and for which $\HH$
is a congruence (so that $\mathcal{CS}\subseteq\mathcal{C}r$). The class
$\mathcal{C}r$ is defined by the $\mathcal{CR}$ equations \eqref{eq:CR} together
with the equation systems defined by $ab\HH axb$ and $ba\HH bxa$:
\begin{multline}
\mathcal{C}r:\,(\forall a,b)\,(\exists x)\colon \\
(a = axa)\wedge(ax=xa)\wedge(ab\HH axb)\wedge(ba\HH bxa).\label{eq:6}
\end{multline}
For supposing that $S$ is a cryptogroup then $S$ is completely regular
and since $\HH$ is a congruence and $a\HH ax$ (as $ax=xa$)
it follows that the additional equations are also satisfied. Conversely
if $S$ satisfies our equations then $S$ is certainly a union of
groups. Suppose that $a\HH c$ in $S$. Then $e=ax=xa$ is the
idempotent in the class $H_{a}=H_{c}$. By the same token there is a solution
$y$ say to the given equations so that $c=cyc$ and $cy = yc=e$.
It then follows from the third equation in (6) 
that for any $b\in S$ we have $ab\HH eb\HH cb$, so that $\HH$ is a right congruence, and by symmetry
we obtain that $\HH$ is also a left congruence, and therefore
$\HH$ is a congruence on $S$, which is to say that $S$ is
a cryptogroup. 
\end{enumerate}
\end{example}

\begin{example}\label{eg:2}
\begin{enumerate}
\item[(i)] \emph{Semigroups with a right identity} (resp.~\emph{right zero})
are defined by the equation 
\begin{equation}
(\exists x)\,(\forall a)\colon ax=a\,(\text{resp.~\ensuremath{ax=x)}}.\label{eqn:rightid}
\end{equation}
We also have of course the left and the two-sided versions of these,
the two-sided cases respectively being the classes of \emph{Monoids}
($\mathcal{M})$, and \emph{Semigroups with zero}. In accord with the
comment above, we may express these respectively via the equations
$x=1$ and $x=0$. We do however explicitly call attention to this
equational basis for $\mathcal{M}:$
\begin{equation}
\text{\ensuremath{\mathcal{M}}:}\ (\exists x)\,(\forall a)\colon ax=xa=a.\label{eq:M}
\end{equation}
The order of the logical quantifiers $\forall$ and $\exists$ in
the equation systems of \eqref{eqn:group} to \eqref{eq:6} is $\forall\dots\exists\dots$ whereas
in \eqref{eqn:rightid} and \eqref{eq:M} the order is reversed. What is more $\mathcal{M}$ cannot
be represented by an equational basis of the form $\forall^{*}\,\exists^{*}$
(meaning any, possibly $0$, number of $\forall$ followed by any,
possibly zero, number of $\exists$) because any class defined in
that way is closed under the taking of unions of ascending chains of algebras (this is the easy half of the Chang-{\L}os-Suszko Preservation Theorem; see \cite[Theorem~5.2.6]{chakei}).
However, this is not true of $\mathcal{M}$ as may be seen by considering
the semilattice represented by the chain $E$  of the integers with the standard ordering.  Observe that $E$  is the union of the ascending sequence of sub-chains  $E_i=\{ e_{-i}< e_{-i+1} < \cdots <  e_{i - 1} < e_{i}\}$  $(i\geq 1)$.  Each $E_i$ is a monoid with zero, with identity element  $e_{i}$  and with zero element  $e_{-i}$  but the union $E$ of these chains is a semilattice with no identity element and no zero.  In particular, $E$  is not a monoid. 

\item[(ii)] The class $\mathcal{J}$ of \emph{simple semigroups} (semigroups
with a single $\JJ$-class) is defined by the
condition that for all $a,b\in S$ there are solutions $x,y\in S^{1}$
to the equation $a=xby$. In those circumstances however, by replacing
$b$ by $aba$ we may find solutions $u,v\in S^{1}$ such that $a=(ua)b(av)$
so that $x=ua$ and $y=av$ furnish solutions $x,y\in S$ that satisfy
our equation $a=xby$. In summary we have the following equational
basis for $\mathcal{J}$: 
\begin{equation}
\text{$\JJ$-simple semigroups:}\ (\forall a,b)\,(\exists x,y)\colon a=xby.
\end{equation}

By a similar argument, the class $\mathcal{R}$ of \emph{right simple semigroups} ($aS^{1} = S$   for all $a \in S$) has equational basis $(\forall a,b)\, (\exists x):\, a = bx$. The dual comment applies to the class  $\mathcal{L}$  of \emph{left simple semigroups}, while the class of $\mathcal{H}$-simple semigroups is of course the class $\mathcal{G}$ of all groups. 
 
\item[(iii)] Semigroups with a maximum  $\JJ$-class.   The two element null semigroup $N$ has a maximum $\JJ$-class, but its square $N\times N$ does not, so the class of semigroups with maximum $\JJ$-class is not $\mathsf{P}$-closed.  It turns out that this example is the main obstacle to being $\{\mathsf{H},\mathsf{P}\}$-closed, as arguments similar to those in part (ii) show that the following properties are equivalent for a semigroup $S$:
\begin{itemize}
\item $S$ has a maximum $\JJ$-class $J$ and the Rees quotient  $S/(S-J)$ is not a null semigroup;
\item $S$ has a maximum $\JJ$-class $J$ and $S/(S-J)$ is \emph{not} isomorphic to $N$;
\item $S$ satisfies equation \eqref{eq:maximum}:
\end{itemize}
\begin{equation}
(\exists y)\,(\forall a)\,(\exists x,z)\colon  (a = xyz). \label{eq:maximum}
\end{equation}

Equation \eqref{eq:maximum} has a sequence of three alternating quantifier types, and we now show that this is necessary.  By the Chang-{\L}os-Suszko Preservation Theorem, it suffices to show that there is a subsemigroup chain $A_1\leq A_2\leq \dots$ such that each $A_i$ satisfies Equation~\eqref{eq:maximum} and such that  $\bigcup_{j\geq 1}A_j$ fails Equation~\eqref{eq:maximum}, and a subsemigroup chain  $B_1\leq B_2\leq\dots$ with each $B_i$  failing Equation \eqref{eq:maximum}, but such that $\bigcup_{j\geq 1}B_j$ satisfies Equation \eqref{eq:maximum}.  For the semigroups $A_i$ we may use the semigroups $E_i$ of Example \ref{eg:2}(i): the union $E$ fails \eqref{eq:maximum}.  For $B_j$, we begin by  considering the denumerably generated combinatorial Brandt semigroup ${\bf B}_\omega$, whose set of elements is $\{0\}\cup\{(i,j)\mid i,j\in \omega=\{0,1,2,\dots\}\}$ with $0$ acting as a multiplicative zero element and with multiplication 
\[
(i,j)(k,\ell)=\begin{cases}
(i,\ell)&\text{ if }j=k\\
0&\text{ otherwise}.
\end{cases}
\]  
For each $i=1,2,3,\dots$, choose $B_i$ to be the subsemigroup of ${\bf B}_\omega$ on the set $\{0\}\cup\{(j,k)\mid 0\leq j,k\leq i-1\}\cup\{(i,i)\}$, with two maximal $\JJ$-classes, which are $\{ (i,i) \}$ and $S \setminus (\{ 0\} \cup \{(i,i)\})$.  Then each $B_i$ fails \eqref{eq:maximum}, yet the union is ${\bf B}_\omega$, which has a single maximum $\JJ$-class, satisfies \eqref{eq:maximum}.

\item[(iv)] A non-example: $\mathcal{D}$ the class of bisimple semigroups.
The distinction between solutions of equations over $S$ and over  $S^{1}$ is important however when it comes to $\DD$,
for it is unique among the five Green's relations in that the class
of $\DD$-simple (\emph{bisimple}) semigroups is closed under
$\mathsf{H}$ but not $\mathsf{P}$. Bisimple semigroups are defined by the following
triple disjunction of equational bases:
\begin{multline}
(\forall a,b)\,(\exists t,u,v,x,y)\colon\\
\big((a=tu)\wedge(t=av)\wedge (t=xb) \wedge (b=yt)\big)\vee
\big((a=xb)\wedge(b=ya)\big)\\\vee\big((a=bx)\wedge(b=ay)\big).\label{eq:bisimple}
\end{multline}
For suppose that $S$ satisfies \eqref{eq:bisimple} and let $a,b\in S$. If the first
equation system in \eqref{eq:bisimple} applies to $a$ and $b$ then $a\RR t\LL b$,
while the second and third sets imply that $a\LL b$ and $a\RR b$
respectively. In any event it follows that $S$ is bisimple. Conversely
let $S$ be any bisimple semigroup and let $a,b\in S$.  Then there
exists $t\in S$ such that $a\RR t\LL b$ and so we may satisfy
the first equation system in \eqref{eq:bisimple} for $a$ and $b$ unless $a=t$ or
$t=b$. If we have that $a=t \neq b$, then $a\LL b$ and the second
equation system in \eqref{eq:bisimple} is solvable. Dually, if $b=t \neq a$ then $a\RR b$
and the third equation system in \eqref{eq:bisimple} can be solved for $a$ and $b$. Finally consider the case where $a = b$.  Clearly we may assume that $|S|\geq 2$  in which case either $|R_{a}|\geq 2$  or $|L_a| \geq 2$.  By symmetry, we need deal only with the first case from which it follows that $\exists x \in S$  such that $ax = a$ whence the third equation system in (11) is then satisfied.  Therefore if $S$ is bisimple then $S$ satisfies \eqref{eq:bisimple}. It follows that
the class of bisimple semigroups is closed under the taking of homomorphic
images but, as we now show, not under the taking of direct products. For this reason, it is not possible to remove the conjunctions in (11). (However, the first equation system of (11) does suffice if we allow ourselves solutions over $S^1$ rather than just $S$.)

Let $X$ be a countable infinite set. The \emph{Baer-Levi semigroup
$B$ }is the subsemigroup of the full transformation semigroup $T_{X}$
consisting of all one-to-one mappings $\alpha:X\rightarrow X$ such
that $|X\setminus X\alpha|$ is infinite. It is well known and easily
verified that $B$ is $\RR$-simple, right cancellative, and
idempotent free; in consequence $B$ is $\LL$-trivial. In particular
it follows that there are no factorizations of the form $a=ta$ in
$B$ (as then $ta=t^{2}a$ whence $t=t^{2}$ by right cancellativity)
or what is the same, $a\not\in Ba$ for all $a\in B$. Therefore $B$
is an example of a bisimple semigroup that satisfies the third equation
set in \eqref{eq:bisimple} but not the other two sets. Its left-right dual, $B^{*}$,
will by symmetry also be bisimple and satisfy the second equation
set in \eqref{eq:bisimple} but not the other two. (As another example of a semigroup
that is left simple, left cancellative, idempotent free and hence
$\RR$-trivial, take the semigroup $S$ of all surjections on
$X$ for which every kernel class is infinite.) The semigroup $B\times B^{*}$
is then an example of a direct product of two $\DD$-simple semigroups
that is not itself $\DD$-simple: indeed $B\times B^{*}$ is
$\DD$-trivial (but $\JJ$-simple), by virtue of the following
observation.

\begin{proposition} Let $S$ \up(resp.~$T$\up) be a semigroup that
satisfies the condition that for all $a\in S$, $a\not\in Sa$ \up(resp.\
for all $b\in T,$ $b\not\in bT$\up). Then $S\times T$ is $\DD$-trivial. 
\end{proposition}

\begin{proof} We first check that $S\times T$ is right trivial.
Suppose that $(a,b)\RR (c,d)$ say. Then either $(a,b)=(c,d)$
or there exists $(x,y),(u,v)\in S\times T$ such that $(a,b)(x,y)=(ax,by)=(c,d)$
and $(c,d)(u,v)=(cu,dv)=(axu,byv)=(a,b)$. But then we have $t=yv\in T$
and $b=bt$, contradicting that $b\not\in bT$. Therefore it follows
that $S\times T$ is right trivial. By symmetry it follows that $S\times T$
is also left trivial and hence $S\times T$ is $\DD$-trivial.
\end{proof}

\item[(v)] Another non-example, this time in the signature of rings with identity, $\{+,-,\cdot,0,1\}$.  Let $F$ be any uncountable field such as the reals or complex numbers.  The closure of $\{F\}$ under taking $\mathsf{H}$ and $\mathsf{P}$ is equal to $\mathsf{HP}(F)$ by the well known class operator inequality $\mathsf{HP}\geq \mathsf{PH}$ (see \cite[Lemma II.9.2]{bursan} for example).  However it cannot be characterised by any set of equation systems, nor indeed by any set of first order sentences, as such a class that is closed under $\mathsf{HP}$ is then an elementary class.  However, we will show that all nontrivial members of $\mathsf{HP}(F)$ have cardinality at least equal to $|F|>\aleph_0$, whereas the Downward Lowenheim-Skolem-Tarski Theorem (see \cite[Corollary 2.1.4]{chakei} for example) shows that, in a countable signature, any elementary class with an infinite model has a model of denumerable cardinality.  It follows that $\mathsf{HP}(F)$ cannot be closed under taking elementary embeddings.

Now there is a natural embedding of $F$ into any nonempty power of itself $F^S$, namely that which assigns each $f\in F$ to the corresponding constant mapping $\underline{f}:s\mapsto f$ in $F^S$.  Let $\underline{F}$ be the image of $F$ under this embedding.  We claim that if two distinct elements of $\underline{F}$ are in the kernel of a homomorphism $\phi$ from $F^S$, then $\phi$ is a trivial homomorphism collapsing all of $F^S$.  To this end, assume $f\neq g$ in $F$ and 
$\phi(\underline{f})=\phi(\underline{g})$.  There is no loss of generality in assuming that $f$ is $1$ and $g$ is $0$, as we may otherwise replace $f$ by $(f-g)^{-1}(f-g)$ and $g$ by $0$.  Then for any $x\in F^S$ we have $\phi(x)=\phi(\underline{1}x)=\phi(\underline{1})\phi(x)=\phi(\underline{0})\phi(x)=\phi(\underline{0}x)=\phi(\underline{0})$ as required.  This completes the argument, as either $\aleph_0<|F|\leq |\phi(F^S)|$ or $|\phi(F^S)|=1$.
\end{enumerate}
\end{example}

\begin{enumerate}

\item[(vi)] Let $\mathcal{RG}$ denote the class of \emph{right groups}, by
which we mean semigroups that are right simple ($aS=S)$ and left
cancellative $((ab=ac)\rightarrow(b=c))$. Another characterization
of $\mathcal{RG}$ is the class of semigroups for which there is always
a \emph{unique} solution to the equation $ax=b$ $(a,b\in S)$. The
solvability of the equation $ax=b$ however does not in itself imply
uniqueness: the Baer-Levi semigroup is an example of a right simple,
right cancellative semigroup in which the equation $ax=b$ always
has infinitely many solutions. However right groups are also characterized
as those semigroups that are right simple and contain an idempotent
(see \cite{clipre} for details) and as such the class is determined by the
availability of solutions to a pair of equations:
\begin{equation}
\text{\ensuremath{\mathcal{RG}}\ensuremath{:}\,}(\forall a,b)\,(\exists x,y)\colon (ax=b)\wedge (y=y^{2}).
\end{equation}

\item[(vii)] Any \emph{variety} $\mathcal{V}$ of semigroups (a class closed under
the operators $\mathsf{H}$, $\mathsf{P}$, and $\mathsf{S}$, the taking of subalgebras) is,
by Birkhoff's theorem, defined by some countable set of identities,
which are equations that may be expressed without the use of the existential
symbol $\exists$.  The following easy proposition is indicative of the kind of result that our approach leads to: it shows for example that the equation systems holding in a variety $\mathcal{V}$ are precisely those holding on the denumerably generated $\mathcal{V}$-free algebra (precisely as is the case for identities).
\end{enumerate}
\begin{proposition}
Let $\mathcal{K}$ be a class of algebras in a countable signature that is defined by a family of equation systems.  Then $\mathcal{K}$ is a variety if and only if~$\mathcal{K}$ contains the denumerably generated $\mathsf{HSP}(\mathcal{K})$-free algebras.
\end{proposition}
\begin{proof}
The forward direction is trivial.  Now assume that the denumerably generated $\mathsf{HSP}(\mathcal{K})$-free algebras lie in $\mathcal{K}$.  Then as $\mathsf{H}(\mathcal{K})=\mathcal{K}$ it follows that $\mathcal{K}$ contains all countably generated members of $\mathsf{HSP}(\mathcal{K})$.  However as $\mathcal{K}$ is defined by a family of equation systems it is an elementary class (definable in first order logic) and hence is determined by its countably generated members.  As these coincide with  $\mathsf{HSP}(\mathcal{K})$ it follows that $\mathsf{HSP}(\mathcal{K})=\mathcal{K}$.
\end{proof}
\begin{enumerate}
\item[(viii)] The dual idea to that which arises in (vii) is of a class defined
by an equation system that is free of the symbol $\forall$. For example
the class: 
\begin{equation}
\text{\ensuremath{\mathcal{I}d}\,:}\ (\exists\,x)\colon x=x^{2}.
\end{equation}
is the class of all semigroups $S$ with idempotents, which is to
say that $E(S)\neq\emptyset$. We note that $\mathcal{I}d$ is the minimum
semigroup class of this kind as any semigroup $S$ with an idempotent
$e$ satisfies every equation $p=q$ that is free of the $\forall$
quantifier as is seen by acting the substitution $x\rightarrow e$
on each variable $x$ of $p=q$. We return to this topic in Section~\ref{sec:universal}.
\end{enumerate}

\section{The equational representation theorem for $\{\mathsf{H},\mathsf{P}\}$-classes}\label{sec:EHP}
A formula of the predicate calculus is in \emph{prenex form} if it
is written as a string of quantifiers (referred to as the \emph{prefix})
followed by a quantifier-free part (referred to as the \emph{matrix}).
An \emph{equation system}, as informally described in Section  \ref{sec:examples}, is a sentence in prenex form, whose matrix
is a conjunction of atomic formulas. Familiar examples include identities
(universally quantified equation systems) and primitive positive sentences
(existentionally quantified equation systems). When the quantifiers
are all universal, we also refer to a $\forall_{1}$ equation system,
while a primitive positive sentence will be referred to as a $\exists_{1}$
equation system. In general, we let a $\forall_{i+1}$ equation system denote an equation system in which the leftmost quantifier is $\forall$, and there are $i$ alternations of quantifiers; the definition of an $\exists_{i+1}$ system is dual. 

We note that in the case of $\forall_{1}$ equation systems, we may
use the property 
\[
\big((\forall x)\,\phi(x)\wedge\psi(x)\big)\leftrightarrow\big((\forall x)\, \phi(x))\wedge\big((\forall x)\,\psi(x)\big)
\]
in order to remove conjunctions (in favour of sets of quantified atomic
formulas), however this is not in general true once existentially
quantified variables are present. Equation systems are exactly the
\emph{positive} (that is, negation-free) \emph{Horn sentences} (sentences
with at most one \emph{positive literal} or \emph{atom}).

The following theorem is an extension of Lyndon's Positivity Theorem
(see~\cite{lyn}, or \cite[Corollary~8.3.5]{hod}), and applies in all signatures,
including those involving relations.  In the case of relations, by
a  \emph{homomorphic image} of a structure $A$, we mean any structure $B$ for which there is a homomorphism from $A$ onto $B$;  the homomorphism does not necessarily map the relations on $A$  onto those of $B$.  
The result is in the style of the many classical preservation theorems of model theory, though does not appear to have been included in standard references such as \cite{chakei} and \cite{hod}, nor in other surveys such as \cite{ros}.  In revision of this manuscript however, the authors have discovered that the result is noted in Keisler \cite{kei} (see un-numbered remark on page 322).  A proof idea is alluded to there (specifically, relating to the proof of Theorem~2 of Bing \cite{bin}), but we feel that the reader will appreciate the transparent inductive argument presented here in full. 
 
\begin{theorem}\label{thm:equationsystem} An elementary class equals the class of models
of some family of equation systems if and only if it is closed under
taking homomorphic images of direct products. If the elementary class
is the model class of a single sentence, then it is a class of models
of a single equation system. 
\end{theorem}
\begin{proof}  One direction is easy: equation systems are preserved
under direct products and under homomorphic images. Note also that  equation systems are examples of first order sentences, so such a class is automatically an elementary class and the initial assumption of the theorem is redundant.  We now must show
that if $\mathcal{K}$ is an elementary class closed under taking surjective
homomorphic images and direct products, then it can be axiomatised
by a family of equation systems. Our method of proof will automatically
derive the second statement in the theorem.  As $\mathcal{K}$ is an elementary class closed under taking homomorphic images, a version of Lyndon's Positivity Theorem applies to show that $\mathcal{K}$ is the class of models of a set $\Sigma$ of positive sentences (see \cite[Theorem~3.2.4]{chakei} or \cite[Exercise 8.3.1]{hod}).  It now remains to show that disjunctions can be removed from the
sentences in $\Sigma$.

Consider a sentence $\rho \in \Sigma$; we assume that $\rho$ is a $\forall_{k}$ sentence,
for, if we are given a $\exists_{k}$ sentence, we may augment $\rho$
with the initial condition $(\forall a)$, where $a$ is a symbol
that does not appear elsewhere in $\Sigma$, and so replace $\rho$
with an equivalent $\forall_{k+1}$ sentence. Therefore we may take
it that the quantifier $Q_{t}$ is $\forall$ if $t$ is odd and $\exists$
if $t$ is even $(1\leq t\leq k)$.   We may write our sentence as:
\begin{multline*}
\rho=(\forall x_{1,1}\dots\forall x_{1,n_{1}})(\exists x_{2,1}\dots\exists x_{2,n_{2}})\dots (Q_{k}x_{k,1}\dots Q_{k}\, x_{k,n_{k}})\\
\rho(x_{1,1},\dots x_{1,n_{1}},x_{2,1},\dots,x_{2,n_{2}},\dots,x_{k,1}\dots,x_{k,n_{k}}).
\end{multline*}
Moreover, there is no loss of generality in assuming that $k$ is
even, as we may, if necessary, append a final $(\exists x)$ quantifier,
where the symbol $x$ does not appear elsewhere in  $\Sigma$, giving
an equivalent sentence. We assume that the matrix of $\rho$ is written
as a finite conjunction of disjunctions; say $\bigwedge_{1\leq i\leq m}\gamma_{i}$,
where each $\gamma_{i}$ is a finite disjunction:
\[
\gamma_{i}=\alpha_{i,1}\vee\dots\vee\alpha_{i,r_{i}}
\]
where each $\alpha_{i,j}$ is an atomic formula involving some subset
of the full set of variables $x_{1,1},\dots,x_{k,n_{k}}$. If $r_{i}=1$
for $i=1,\dots,m$ then there is nothing to prove. Otherwise, if
there is $i$ such that $r_{i}\geq2$, we shall show that there is
a $j\in\{1,\dots,r_{i}\}$ such that the conjunct $\gamma_{i}$ may
be replaced by the single atomic formula $\alpha_{i,j}$. Repeating
this for each conjunct will see us arrive at the desired $\vee$-free
sentence. The quantifiers remain unchanged throughout. 

Without loss of generality then, we may assume that $r_{i}\geq2$
for some $i$. For each $j=1,\dots,r_{i}$ let $\rho_{j}$ be the
result of replacing $\gamma_{i}$ by $\alpha_{i,j}$ in $\rho$. Note
that $\rho_{j}\vdash\rho$ so that the class of models satisfied by
$(\Sigma\cup\{\rho_{j}\})\setminus\{\rho\}$ is a subclass of $\mathcal{K}$.
We wish to show that there is some $j$ such that the reverse containment
holds. 

Assume by way of contradiction that no such $j$ exists. In this case,
for each $j\in\{1,\dots,r_{i}\}$ there is a model $\boldsymbol{M}_{j}\in\mathcal{K}$
such that $\rho_{j}$ fails in $\boldsymbol{M}_{j}$. We will now
use the fact that $\boldsymbol{M}:=\Pi_{1\leq j\leq r_{i}}\boldsymbol{M}_{j}\in\mathcal{K}$
and so $\boldsymbol{M}\models\rho$ in order to produce the required
contradiction. 

For each $j$: as $\boldsymbol{M}_{j}\not\models\rho_{j}$ there is
an $n_{1}$-tuple $a_{1,1,j},\dots,a_{1,n_{1},j}$ such that 
\begin{multline*}
\boldsymbol{M}_{j}\not\models(\exists x_{2,1}\dots\exists x_{2,n_{2}})\dots (\exists x_{k,1}\dots \exists x_{k,n_{k}})\\
\rho_{j}(a_{1,1,j},\dots,a_{1,n_{1},j},x_{2,1},\dots,x_{2,n_{2}},\dots,x_{k,1},\dots,x_{k,n_{k}}),
\end{multline*}
where the final block of quantifiers is $\exists$, from our convention that $k$ is even.
Equivalently
\begin{multline}
\boldsymbol{M}_{j}\models(\forall x_{2,1}\dots \forall x_{2,n_{2}})\dots (\forall x_{k,1}\dots \forall x_{k,n_{k}})\\
\neg\rho_{j}(a_{1,1,j},\dots,a_{1,n_{1},j},x_{2,1},\dots,x_{2,n_{2}},\dots,x_{k,1},\dots,x_{k,n_{k}}).\label{eq:negrho}
\end{multline}

Now let $a_{1,1},\dots,a_{1,n_{1}}\in\boldsymbol{M}$ be the $r_{i}$-tuples
formed from these violating tuples from each $\boldsymbol{M}_{j}$,
which is to say that
\begin{equation}
a_{1,l}(j)=a_{1,l,j}\,\,(1\leq l\leq n_{1}).\label{eq:14}
\end{equation}
Now $\boldsymbol{M}\vDash\rho$, and so there exist elements $a_{2,1},\dots,a_{2,n_{2}}\in\boldsymbol{M}$
such that
\begin{multline}
\boldsymbol{M}\vDash(\forall x_{3,1}\dots\forall x_{3,n_{3}})\dots (\exists x_{k,1}\dots \exists x_{k,n_{k}})\\
\rho(a_{1,1},\dots,a_{1,n_{1}},a_{2,1},\dots,a_{2,n_{2}},x_{3,1},\dots,x_{3,n_{3}},\dots,x_{k,1},\dots,x_{k,n_{k}}).
\label{eq:15}
\end{multline}
We continue inductively in this way and assume that for some $t\geq1$,
for each $j\in\{1,\dots,r_{i}\}$ there exists elements 
\[
a_{1,1,j},\dots,a_{1,n_{1},j},\dots,a_{2t-1,1,j},\dots,a_{2t-1,n_{2t-1},j}\in\boldsymbol{M}_{j}
\]
such that
\begin{multline}
\boldsymbol{M}_{j}\vDash(\forall x_{2t,1}\dots\forall x_{2t,n_{2t}})\dots (\exists x_{k,1}\dots \exists x_{k,n_{k}})\\
\neg\rho_{j}(a_{1,1,j},\dots,a_{1,n_{1},j},\dots,a_{2t-1,1,j},\dots,a_{2t-1,n_{2t-1,}j},\\
x_{2t,1},\dots,x_{2t,n_{2t},}\dots,x_{k,1},\dots,x_{k,n_{k}}).\label{eq:manylines}
\end{multline}
And with
\begin{equation}
a_{m,l}(j)=a_{m,l,j}\,\,(1\leq l\leq n_{p})\,(1\leq p \leq2t-1)\label{eq:aml}
\end{equation}
there exist elements $a_{2t,1},\dots,a_{2t,n_{2t}}\in\boldsymbol{M}$
such that with
\begin{multline}
\boldsymbol{M}\vDash(\forall x_{2t+1,1}\dots\forall x_{2t+1,n_{2t+1}})\dots (\exists x_{k,1}\dots \exists x_{k,n_{k}})\\
\rho(a_{1,1},\dots,a_{1,n_{1}},\dots,a_{2t,1},\dots,a_{2t,n_{2t}},x_{2t+1,1},\dots\\\dots,x_{2t+1,n_{2t+1}}\dots,x_{k,1},\dots,x_{k,n_{k}}).\label{eq:twolines}
\end{multline}
The base $t=1$ case of \eqref{eq:manylines}, \eqref{eq:aml}, and \eqref{eq:twolines} is given by \eqref{eq:negrho}, \eqref{eq:14}
and \eqref{eq:15} respectively. We now verify that we may increment each of
the three parts of the inductive hypothesis, they being \eqref{eq:manylines}, \eqref{eq:aml}, and \eqref{eq:twolines}, from $t$ to $t+1$ and thereby continue the induction. 

We use \eqref{eq:twolines} to project from $\boldsymbol{M}$ to each $\boldsymbol{M}_{j}$
by making substitutions in~\eqref{eq:manylines}: 
\begin{equation}
x_{2t,1}\mapsto a_{2t,1,j}=a_{2t,1}(j),\dots,x_{2t,n_{2t}}\mapsto a_{2t,n_{2},j}=a_{2t,n_{2}}(j).\label{eq:19}
\end{equation}
Then from  \eqref{eq:manylines} we have: 
\begin{multline}
\boldsymbol{M}_{j}\vDash(\exists x_{2t+1,1}\dots\exists x_{2t+1,n_{2t+1}})\dots (\exists x_{k,1}\dots \exists x_{k,n_{k}})\\
\neg\rho_{j}(a_{1,1,j},\dots,a_{1,n_{1},j},\dots,a_{2t,1,j},\dots,a_{2t,n_{2t},j},\\
x_{2t+1,1},\dots,x_{2t+1,n_{2t+1},}\dots,x_{k,1},\dots,x_{k,n_{k}}).\label{eq:20}
\end{multline}
Substituting witnesses $x_{2t+1,l,j}\mapsto a_{2t+1,l,j}$ $(1\leq l\leq n_{2t+1})$
in \eqref{eq:20} then increments \eqref{eq:manylines} from $t$ to $t+1$:
\begin{multline}
\boldsymbol{M}_{j}\vDash(\forall x_{2t+2,1}\dots\forall x_{2t+2,n_{2t+2}})\dots (\exists x_{k,1}\dots \exists x_{k,n_{k}})\\
\neg\rho_{j}(a_{1,1,j},\dots,a_{1,n_{1},j},\dots,a_{2t+1,1,j},\dots,a_{2t+1,n_{2t+1},j},\\
x_{2t+2,1},\dots,x_{2t+2,n_{2t+2},}\dots,x_{k,1},\dots,x_{k,n_{k}}).\label{eq:21}
\end{multline}
Next we put $a_{2t+1,l}(j)=a_{2t+1,l,j}\,(1\leq l\leq n_{2t+1})$,
which, together with \eqref{eq:19}, increments \eqref{eq:aml} from $t$ to $t+1$. Finally,
by \eqref{eq:twolines} we may substitute in $\boldsymbol{M}$: 
\[
x_{2t+1,1}\mapsto a_{2t+1,1},\dots,x_{2t+1,n_{2t+1}}\mapsto a_{2t+1,n_{2t+1}},
\]
and call up witnesses: 
\[
x_{2t+2,1}\mapsto a_{2t+2,1},\dots,x_{2t+2,n_{2t+2}}\mapsto a_{2t+2,n_{2t+2}}
\]
such that
\begin{multline}
\boldsymbol{M}\vDash(\forall x_{2t+3,1}\dots\forall x_{2t+3,n_{2t+3}})\dots (\exists x_{k,1}\dots \exists x_{k,n_{k}})\\
\rho(a_{1,1},\dots,a_{1,n_{1}},\dots,a_{2t+2,1},\dots,a_{2t+2,n_{2t+2}},\\
x_{2t+3,1},\dots,x_{2t+3,n_{2t+3}},\dots,x_{k,1},\dots,x_{k,n_{k}})
\end{multline}
which increments \eqref{eq:twolines} from $t$ to $t+1$, and so the induction continues.
This recursive procedure eventually yields a tuple 
\begin{equation}
\bar{a}=(a_{1,1},\dots,a_{1,n_{1}},\dots,a_{k,1},\dots,a_{k,n_{k}})\label{eq:tuple}
\end{equation}
such that $\boldsymbol{M}\vDash\rho(\bar{a})$ but for each $j$,
$(1\leq j\leq r_{i})$, $\boldsymbol{M}_{j}\vDash\neg\rho_{j}(\bar{a}_{j})$,
where $\bar{a}_{j}$ represents the tuple obtained from \eqref{eq:tuple} by projecting
onto the $j$th co-ordinate:
\[
\bar{a}_{j}=(a_{1,1,j},\dots,a_{1,n_{1},j},\dots,a_{k,1,j},\dots,a_{k,n_{k},j}).
\]

Now for all $i'=1,\dots,k$, we have that $\gamma_{i'}(\bar{a})$
is true in \textbf{$\boldsymbol{M}$ }and also $\gamma_{i'}(\bar{a}_{j})$
holds in each $\boldsymbol{M}_{j}$. Now $\boldsymbol{M}_{j}\vDash\neg\rho_{j}(\bar{a}_{j})$
and for $i'\neq i$ the conjunct $\gamma_{i'}$ appears in $\rho_{j}$;
as we have noted, $\boldsymbol{M}_{j}\vDash\gamma_{i'}(\bar{a}_{j})$,
and so it follows that $\alpha_{i,j}(\bar{a}_{j})$ must be false
in $\boldsymbol{M}_{j}$. But as $\gamma_{i}(\bar{a})$ is true in
$\boldsymbol{M}$, there must exist $j\in\{1,\dots,r_{i}\}$ with
$\alpha_{i,j}(\bar{a})$ true. But then, we obtain the contradiction
that $\alpha_{i,j}(\bar{a}_{j})$ is true in $\boldsymbol{M}_{j}$.
Arrival at this contradiction completes the proof. \end{proof}

Because the class operators $\mathsf{H}$ and $\mathsf{P}$ are related in composition by $\mathsf{PH}\leq \mathsf{HP}$, Theorem \ref{thm:equationsystem} can be re-expressed as stating that an elementary class~$\mathcal{K}$ is definable by an equation system if and only if $\mathcal{K}=\mathsf{HP}(\mathcal{K})$.  
If we wish to drop the assumption that $\mathcal{K}$ is an elementary class, we need more care.  
Elementary classes are those closed under taking ultraproducts and elementary embeddings, however in the presence of $\mathsf{H}$ and $\mathsf{P}$, we may ignore ultraproducts because they are particular cases of applications by $\mathsf{HP}$.  We cannot however ignore elementary embeddings, as Example \ref{eg:2}(v) demonstrates.
Thus Theorem~\ref{thm:equationsystem} can be rephrased as ``\emph{a class $\mathcal{K}$ is the class of models of some equation systems if and only if it is closed under $\mathsf{E}$, $\mathsf{H}$ and $\mathsf{P}$}'', where $\mathsf{E}$ denotes closure under taking elementary embeddings.  
In addition to the aforementioned containment $\mathsf{PH}\leq \mathsf{HP}$, it is possible to show that $\mathsf{HE}\leq \mathsf{EHP}$, which points toward the composite $\mathsf{EHP}$ as being a single closure operator equivalent to iterated closure under combinations of $\mathsf{E}$, $\mathsf{H}$ and $\mathsf{P}$.  
Unfortunately the authors are not aware of a useful containment between $\mathsf{PE}$ and $\mathsf{EHP}$.  We refer simply to $\{\mathsf{E},\mathsf{H},\mathsf{P}\}$-closed classes and even $\{\mathsf{E},\mathsf{H},\mathsf{P}\}$-classes, though $\{\mathsf{E},\mathsf{H},\mathsf{P}\}^*$-closed may be more technically correct.
An interesting consequence of Theorem \ref{thm:equationsystem} is that all equationally defined classes arise as reducts of varieties.  This is of course well-known for inverse semigroups and groups (as semigroups), but is not otherwise immediately obvious for other $\{\mathsf{E},\mathsf{H},\mathsf{P}\}$-classes.

The class of reducts of a variety is always closed under ultraproducts and direct products, but in general need not be closed under taking homomorphic images, nor subalgebras, nor even elementary embeddings.  There are plentiful easy examples demonstrating the failure of the first two of these closure properties.  For the case of elementary embeddings, we observe that real vector spaces form a variety (with vector addition as binary and  $\mathbb{R}$-many unary operations for scalar multiplication).  The class of reducts to the empty signature has no countably infinite members, and hence is not an elementary class.  
When the class of reducts of members of a variety is closed under $\mathsf{H}$ and $\mathsf{E}$ (as they are for groups and inverse semigroups), then Theorem \ref{thm:equationsystem} shows that the class is definable by the equation systems.  We now show a converse to this statement.

\begin{thm}\label{thm:reduct}
Let $\mathscr{L}$ be a signature and $\mathcal{K}$ an $\{\mathsf{E},\mathsf{H},\mathsf{P}\}$-closed class of $\mathscr{L}$-structures.  Then $\mathcal{K}$ is the class of reducts of a variety $\mathcal{V}$ in a signature extending $\mathscr{L}$.  If $\mathcal{K}$ is finitely axiomatisable in first order logic, then $\mathcal{V}$ can be chosen to be finitely based and have only finitely many new operations in comparison to $\mathcal{K}$.
\end{thm}
\begin{proof}
As $\mathcal{K}$ is closed under taking homomorphic images, direct products and elementary embeddings, it can be axiomatised by a family of equation systems by Theorem \ref{thm:equationsystem}.  If $\mathcal{K}$ is finitely axiomatisable in first order logic, then the Completeness Theorem for first order logic ensures that it can equivalently be axiomatised by a finite family of equation systems.  We now explain how to replace each equation system in the family of equation systems by an identity in some extended signature.  A finite number of new operations is added for each equation system, so that if the family of equation systems defining $\mathcal{K}$ is finite, then so also will the resulting variety (after all equation systems are replaced) be finite. 

Consider a signature $\mathscr{L}$ and a sentence $(\forall\vec{x})\,(\exists y)\, \phi(\vec{x},y)$ in $\mathscr{L}$, where $\vec{x}$ abbreviates $x_1,\dots,x_n$ (for some $n$, possibly $0$) and where $\phi(\vec{x},y)$ may contain other quantified variables that are not displayed.  Let $f$ be a new $n$-ary operation symbol.  It easily verified that the models of $(\forall\vec{x})\,(\exists y)\, \phi(\vec{x},y)$ are precisely the $\mathscr{L}$-reducts of models of $(\forall\vec{x})\, \phi(\vec{x},f(x_1,\dots,x_n))$.  Indeed, if ${\bf M}\models (\forall\vec{x})\,(\exists y)\, \phi(\vec{x},y)$ then we may expand the signature $\mathscr{L}$ of ${\bf M}$ to include $f$ by defining $f$ at each tuple $\vec{a}\in  {\bf M}^n$ to be any witness $x$ to $(\exists x)\,  \phi(\vec{a},x)$.  This expansion of ${\bf M}$ is not necessarily unique, but all such expansions are models of $(\forall\vec{x})\, \phi(\vec{x},f(x_1,\dots,x_n))$.  Conversely the $\mathscr{L}$-reduct of any model ${\bf N}$ of $(\forall\vec{x})\,\phi(\vec{x},f(x_1,\dots,x_n))$ will be a model of $(\forall\vec{x})\,(\exists y)\, \phi(\vec{x},y)$, as the value of~$f$ at any tuple $\vec{a}\in {\bf N}^n$ provides the witness $x$ required in $(\exists x)\, \phi(\vec{a},x)$.
This process is known as \emph{Skolemisation}; see \cite[\S3.3]{chakei} or \cite[\S3.1]{hod} for example.  An application Skolemisation to an equation system produces an equation system with one fewer existential quantifiers.  Repeated applications eventually yield an equation system without any existential quantifiers.  Such an equation system is a finite set of identities.
\end{proof}
As a first example, Skolemising the defining equation $(\exists x)(\forall a)\colon xa=ax=a$ for monoids (as semigroups), we introduce a nullary operation $e$ to replace $x$ to obtain $(\forall a)\colon ea=ae=a$, the familiar definition as a semigroup with constant.
As a second example, we consider the result of applying Skolemisation to the definition \eqref{eqn:group} for the class of groups as semigroups.  The given sentence is $(\forall a\forall b)(\exists x\exists y)\colon (ax=b)\wedge (ya=b)$.  Skolemising once (using the symbol $\backslash$ for the introduced Skolem function) we obtain $(\forall a\forall b)(\exists y)\colon a(a\backslash b)=b\wedge ya=b$, and then a second time (using $/$) we obtain $(\forall a\forall b)\colon a(a\backslash b)=b\wedge (b/a)a=b$ (note that $b/a$ might have more consistently been written as $a/b$, however it is immediately clear that the required value is the element $ba^{-1}$).  Thus groups (as semigroups) are the class of reducts of the variety with two additional binary operations $\backslash,/$ defined by the identities $a(a\backslash b)=b$ and $(b/a)a=b$ in addition to associativity of the semigroup multiplication.
  
\section{Equational bases for e-varieties of semigroups }\label{sec:evar}

The theory of semigroup e-varieties was devised by T.E. Hall \cite{hal89}
and others in order create an interface between the theory of regular
semigroups and universal algebra. The theory endows $\{\mathsf{H},\mathsf{P}\}$-closed classes
of regular semigroups with the structure of a varietal type through
introduction of unary operations corresponding to choices of inverses
for elements of the regular semigroups involved. From Section \ref{sec:EHP} we
know that we may, at least in principle, represent such classes by
equational bases using only the associative binary operation with
which the semigroup is naturally endowed. We shall henceforth refer
to these as  $\{\mathsf{E},\mathsf{H},\mathsf{P}\}$-bases, indicating that they determine a defining set
of equations for a class that is closed under $\mathsf{E}$, $\mathsf{H}$ and $\mathsf{P}$. 

Three fundamental classes that form e-varieties are the classes $\mathcal{I}$
of all \emph{inverse semigroups}, $\mathcal{O}$ of \emph{orthodox semigroups},
and $\mathcal{ES}$ of so-called \emph{$E$-solid semigroups}. In this
section we obtain finite equational bases for these benchmark e-varieties.
The bases each involve choosing inverses for two arbitrary members
$a,b\in S$ and adjoining a second set of equations that ensure that
all products of the associated idempotents of a certain length (length
$3$ for $\mathcal{ES}$ and for $\mathcal{O}$, and $2$ for $\mathcal{I}$)
have a particular property associated with the class (are group members
for $\mathcal{ES}$, are idempotents for $\mathcal{O}$, and commute with
one another in the case of $\mathcal{I}$). 

We begin with the equational basis problem for the class $\mathcal{ES}$
of all \emph{$E$-solid} \emph{semigroups}, which are those regular
semigroups $S$ that satisfy the \emph{solidity condition} that for
idempotents $e,f,g\in E(S)$ 
\begin{equation}
e\LL f\RR g\rightarrow\exists h\in E(S)\colon e\RR h\LL g.
\label{eq:solidity}
\end{equation}
Note that $\mathcal{\mathcal{O}\subseteq ES}$ and $\mathcal{ES\cap IG=CR}$.
We make use of the fact, taken from~\cite{hal73}, that a semigroup $S$
is $E$-solid if and only if $S$ is a regular semigroup for
which the idempotent-generated subsemigroup $\langle E(S)\rangle$
is a union of groups.  

The equational basis that we use here for the class of E-solid semigroups consists of the following equations:
\begin{multline}
(\forall a,b)\,(\exists x,y,z_1,z_2, \cdots , z_{64})\colon x \in V(a) \wedge y \in V(b)\\
\ \wedge z_1 \in V(ax\cdot ax \cdot ax) \wedge z_1(ax\cdot ax \cdot ax) = (ax \cdot ax \cdot ax)z_1 \\
\ \wedge z_2 \in V(ax\cdot ax \cdot xa) \wedge z_2(ax\cdot ax \cdot xa) = (ax \cdot ax \cdot xa)z_2 \\
\ \vdots \\
\ \wedge z_{64} \in V(yb\cdot yb \cdot yb) \wedge z_{64}(yb \cdot yb \cdot yb) = (yb \cdot yb \cdot yb)z_{64}
\end{multline}
where the last $64$ lines run over all products of three (not necessarily distinct) members of the set $F = \{ax,xa,by,yb\}$.  There are redundancies here as some subsets of the collection are adequate equational bases for the class of E-solid semigroups but the above array is convenient for the uniformity of  presentation that it offers.  We may express this more succintly however by writing this basis as follows:
\begin{equation*}
(\forall a\forall b)\,(\exists x\exists y)\colon (x\in V(a))\wedge (y\in V(b))\wedge (p\in G).
\end{equation*}
where $p$ denotes the product $g_{1}g_{2}g_{3}$ for $g_1,g_2,g_3\in F$.  

A basic theorem of semigroups due to Miller and Clifford \cite{clipre} is that in any semigroup $S$, if $x \LL a \RR  y$  with $a$ a member of a subgroup of $S$ then $x \RR xy \LL y$, or in other words $xy \in R_x \cap L_y$.    This fact will be used frequently in the proofs of this section without further reference. 

\begin{theorem}\label{thm:ESolid}
The class $\mathcal{ES}$ of all $E$-solid semigroups
$S$ is the $\{\mathsf{E},\mathsf{H},\mathsf{P}\}$-class with $\{\mathsf{E},\mathsf{H},\mathsf{P}\}$-basis\up:
\begin{equation}
(\forall a,b)\,(\exists x,y)\colon (x\in V(a)\wedge y\in V(b))\wedge \bigwedge_{p=g_1g_2g_3}(p\in G)
\label{eq:Esolid}
\end{equation}
where $g_1,g_2,g_3\in F=\{ax,xa,by,yb\}$.
\end{theorem}
\begin{proof} 

If $S$ is E-solid, then $S$ is regular and so the first part of \eqref{eq:Esolid} holds: for every $a,b$ we can find $x\in V(a)$ and $y\in V(b)$. Moreover, since $\langle E(S) \rangle $ is a union
of groups, for any product $p$ of idempotents, we have that the condition $p \in G$  is satisfied. Thus, since $F = \{ax, xa, by, yb\} \subseteq  E(S)$, it follows that the whole equation system \eqref{eq:Esolid} holds.

Conversely, suppose that $S$ satisfies the equation system 
\eqref{eq:Esolid}; the first conjuncts of \eqref{eq:Esolid} ensure that $S$ is regular. Take $a,b\in E(S)$ and take witnesses $x,y\in S$ for satisfaction of~\eqref{eq:Esolid}. Since $x = xax = xa\cdot ax \in F^2 \subseteq F^3$, \eqref{eq:Esolid} tells us that~$H_x$ is a group and $x\in E(S)$. Moreover $ax \cdot xa = axa = a$.  Similarly,~$H_y$ is a group, $y\in E(S)$, $yb \cdot by = y$  and $by \cdot  yb = b$.

Now suppose we have $f \in E(S)$  such that $a \LL f \RR b$. We need to show that $a \RR g \LL b$
for some idempotent $g$, or equivalently that $R_a \cap L_b = H_{ab}$  is a group (refer to the $\DD$-class diagram below).  
\begin{center}
\begin{tabular}{|c|c|c|c|c|}
\hline 
$x$ & $xa$ &  & $xa \cdot b$  & \tabularnewline
\hline 
$ax$ & $a$ &  & $a \cdot yb$  & \tabularnewline
\hline 
 &  & $\cdots$  &  & \tabularnewline
\hline 
 & $f$ &  & $b$ & $by$\tabularnewline
\hline 
 & $yb \cdot xa$ &  & $yb$ & $y$\tabularnewline
\hline 
\end{tabular} 
\end{center}

Observe now that $xa \cdot b = xa \cdot by \cdot yb$  so that $H_{xa\cdot b} = R_{xa} \cap L_b$  is a group. It follows in turn that $R_{yb} \cap L_{xa} = H_{yb  \cdot xa}$ is also a group. Finally we infer that $H_{ab} = H_{a \cdot yb} = H_{ax \cdot xa \cdot yb}$  is too a group, as required.  \end{proof}

%
%

\begin{theorem} \label{thm:orthodox}
The class $\mathcal{O}$ of orthodox semigroups
$S$ is the $\{\mathsf{E},\mathsf{H},\mathsf{P}\}$-class with $\{\mathsf{E},\mathsf{H},\mathsf{P}\}$-basis\up: 
\begin{equation}
(\forall a,b)\,(\exists x,y)\colon (x\in V(a)\wedge y\in V(b))\wedge \bigwedge_{p=g_1g_2g_3}(p\in E(S))\label{eq:orthodox}
\end{equation}
where $g_1,g_2,g_3 \in F=\{ax,xa,by,yb\}$.
\end{theorem}
\begin{proof} Clearly, since the members of $F$ are idempotents,
the equation system \eqref{eq:orthodox} is satisfied by any orthodox semigroup
$S$. Conversely suppose that $S$ satisfies \eqref{eq:orthodox}. Let $a\in E=E(S)$
and take any $b\in V(a)$. We shall show that $b\in E(S)$. Since
in any regular semigroup $V(E)=E^{2}$, it follows from this that
$E=E^{2}$, which is to say that $S$ is orthodox. With $x,y$ witnessing satisfaction of \eqref{eq:orthodox} for the chosen $a,b$, we have $x=xax=xa\cdot ax\in E$, from which it follows that $ax\cdot xa=axa=a$. We have the
following $\DD$-class diagram, the remaining entries of which
are explained below. 
\begin{center}
\begin{tabular}{|c|c|c|c|c|}
\hline 
$x$ & $xa$ &  &  & $xa \cdot by$\tabularnewline
\hline 
$ax$ & $a$ & & $ab$ & $aby$\tabularnewline
\hline 
 &  & $\cdots $ &  & \tabularnewline
\hline 
 & $ba$ &  & $b$ & $by$\tabularnewline
\hline 
 & $yba$ &  & $yb$ & $y$\tabularnewline
\hline 
\end{tabular} \end{center}

Since $ba\in E(S)$ we have $xa\cdot by\in E(S)$ placed as shown;
hence $by\cdot xa=ba$. Similarly since $ab\in E(S)$ we have $yb\cdot a=yb\cdot ax\cdot xa\in E(S)$
as shown. Since  $a\cdot yb=ax\cdot xa\cdot yb\in E(S)$ it now follows that $a\cdot yb=ab.$
Putting these factorizations together yields:
\[
b=b\cdot ab=b\cdot ayb=ba\cdot yb=by\cdot xa\cdot yb\in E(S),
\]
which completes the proof. 
\end{proof}

\begin{theorem}\label{thm:inversesgp}  The class $\mathcal{I}$ of inverse semigroups $S$
is the $\{\mathsf{E},\mathsf{H},\mathsf{P}\}$-class with $\{\mathsf{E},\mathsf{H},\mathsf{P}\}$-basis\up: 
\begin{equation}
(\forall a,b)\,(\exists x,y)\colon (x\in V(a)\wedge y\in V(b))\wedge\bigwedge_{g_{1},g_{2}\in \{ax,xa,by,yb\}}
(g_{1}g_{2}=g_{2}g_{1}).\label{eq:inversedef}
\end{equation}
\end{theorem}
\begin{proof} Clearly any inverse semigroup satisfies \eqref{eq:inversedef}
so only the converse is in question. By the first part of \eqref{eq:inversedef} we have that $S$ is regular so let
us take $a,b\in E(S)$. (We require only three of the six equations
specified by the final group of conjuncts in \eqref{eq:inversedef}). From \eqref{eq:inversedef} we obtain $x = xax = xa \cdot ax = ax \cdot xa$, 
whence $x = axa = a$, and so $a = a^2 = ax$. Similarly, $b = yb$. Then from (29), $ab = ax \cdot yb = yb \cdot ax = ba$.
Since $S$ is regular and every pair of idempotents of $S$ commute,
we have proved that $S$ an inverse semigroup. \end{proof}


\begin{example}
The products of length $3$ in Theorems \ref{thm:ESolid} and \ref{thm:orthodox} cannot be replaced by products of length 2 as in Theorem \ref{thm:inversesgp}.
\end{example}
\begin{proof} 
We demonstrate this by finding an assignment
of an inverse to each member of a specific regular semigroup $S$ in such a way that $S$ satisfies the length-two version of equation system~\eqref{eq:orthodox}:
\begin{equation}
(\forall a,b)\,(\exists x,y)\colon (x\in V(a)\wedge y\in V(b))\wedge \bigwedge_{p=g_1g_2}(p\in E(S)),\label{eq:ength2}
\end{equation}
where $g_{1},g_2\in F=\{ax,xa,by,yb\}$.
However, $S$ is not $E$-solid (and so also not orthodox). 

Let $S$ be the $0$-rectangular band with non-zero $\DD$-class
$D$ defined by the first `eggbox' of the following three diagrams,
where an asterisk denotes an idempotent: 

\begin{center}
\begin{tabular}{|c|c|c|c|}
\hline 
{*} & {*} &  & {*}\tabularnewline
\hline 
{*} & {*} &  & \tabularnewline
\hline 
 & {*} &  & {*}\tabularnewline
\hline 
 &  & {*} & \tabularnewline
\hline 
\end{tabular}~~~~~~~%
\begin{tabular}{|c|c|c|c|}
\hline 
 & 1 &  & \tabularnewline
\hline 
1 &  &  & \tabularnewline
\hline 
 &  &  & 1\tabularnewline
\hline 
 &  & 1 & \tabularnewline
\hline 
\end{tabular}~~~~~~~%
\begin{tabular}{|c|c|c|c|}
\hline 
2 & 2 &  & 2\tabularnewline
\hline 
2 & 2 &  & \tabularnewline
\hline 
 & 2 &  & 2\tabularnewline
\hline 
 &  & 2 & \tabularnewline
\hline 
\end{tabular} \end{center}

We see that $S$ is regular but not $E$-solid since, for example,
the entry at position $(2,4)$ is not idempotent despite the presence
of idempotents at positions $(2,2),\,(3,2)$, and $(3,4)$. We shall
write $(i,j)$ $(1\leq i,j\leq4)$ to denote the element of $D$ in
that corresponding position in the diagram. Let $a\in D$. We shall
choose $x\in V(a)$ writing this selection in the form $a\rightarrow x$.
We assign $(1,1)\leftrightarrow(2,2)$,~$(1,2)\rightarrow(1,2),(2,1)\rightarrow(2,1)$,
$(3,3)\leftrightarrow(4,4)$, $(3,4)\rightarrow(3,4)$, $(4,3)\rightarrow(4,3)$;
$(1,3)\leftrightarrow(4,2)$, $(1,4)\leftrightarrow(3,2)$; $(2,3)\leftrightarrow(4,1)$,
$(2,4)\leftrightarrow(3,1)$. 

In the second diagram the idempotent products of the form $ax$ or
$xa$ (under the previous assignment) are indicated by the numeral
$1$, which are all starred in the first diagram.

In the third diagram the non-zero products of pairs of idempotents
of the form $ax$ or $yb$ $(a,b\in S)$ are indicated by the numeral
$2$. Since each instance of the numeral $2$ lies in a starred square
in the first diagram, it follows that, with the given assignment of
inverses, $S$ satisfies the equations of ~\eqref{eq:ength2} but $S$ is
not $E$-solid. Therefore the equational basis given by ~\eqref{eq:ength2}
does not imply $E$-solidity.\end{proof}

\section{Comparing e-varieties to \mbox{$\{E,H,P\}$}-classes of
regular semigroups}
As mentioned in the introduction, Hall \cite{hal89} initiated the study
of e-varieties of regular semigroups, which are classes $\mathcal{C}$
of regular semigroups closed under~$\mathsf{H}$, the taking of homomorphic
images,~$\mathsf{P}$, the taking of direct products, and~$\mathsf{S_{e}}$, the taking
of regular subsemigroups. In this section we look at the connection
between e-varieties and $\{\mathsf{E},\mathsf{H},\mathsf{P}\}$-classes of regular semigroups. 

\vspace{.2cm}

\begin{theorem} Any e-variety $\mathcal{C}$ of regular semigroups
is an $\{\mathsf{E},\mathsf{H},\mathsf{P}\}$-class of regular semigroups but in general the
converse does not hold.
\end{theorem}
\begin{proof} Certainly $\mathcal{C}$ is a class of regular semigroups
closed under the operators $\mathsf{H}$ and $\mathsf{P}$. Moreover, if $S\in\mathcal{C}$
and $U$ is a subsemigroup that is elementarily embedded in $S$ then
$U$ is also regular as regularity is a first order property. Therefore
$U\in\mathcal{C}$ and so $\mathcal{C}$ is an $\{\mathsf{E},\mathsf{H},\mathsf{P}\}$-class
of regular semigroups. 

It is easy to check that the class $\mathcal{M}\mathcal{R}$ of all regular
monoids is an $\{\mathsf{E},\mathsf{H},\mathsf{P}\}$-class that is not an e-variety but this
example could however be accommodated as an e-variety of regular monoids.
For an example that is not a monoid class and which involves equations
only of the type $\forall\,\cdots\,\exists$, consider the more general
$\{\mathsf{E},\mathsf{H},\mathsf{P}\}$-class:
\[
\mathcal{LR_{\rm reg}}\colon (\forall a)(\exists x,y)\colon a=axa=ya=ay.
\]
Then $\mathcal{LR_{\rm reg}}$ is a class of regular semigroups but, as we shall
verify, $\mathcal{LR_{\rm reg}}$ is not closed under $\mathsf{S_{e}}$. First note that
$\mathcal{MR}\subseteq\mathcal{LR_{\rm reg}}$ and in particular the full transformation
semigroup $T_{3}\in\mathcal{LR_{\rm reg}}$. Next observe that the subsemigroup
$U$ of $T_{3}$ that consists of all mappings that are not permutations
is a regular subsemigroup of $T_{3}$. Now take the mapping $a\in U$
defined by the action $1\mapsto2\mapsto3\mapsto3$ and suppose that
$y\in T_{3}$ is such that $a=ay=ya$. Then since $a=ay,$ $y$ acts
identically on the range of $a$ so that $2y=2$ and $3y=3$. On the
other hand, since $a=ya$ we cannot have $1y>1$ for then $1ya\in\{2,3\}a=3\neq2=1a$.
Hence $1y=1$ and so $y$ is the identity mapping. In particular,
$y\not\in U$ and therefore $U\not\in\mathcal{LR_{\rm reg}}$, thereby showing
that $\mathcal{LR_{\rm reg}}$ is not an e-variety, despite $\mathcal{LR_{\rm reg}}$ being
an $\{\mathsf{E},\mathsf{H},\mathsf{P}\}$-class of regular semigroups. \end{proof}

The e-variety approach involves the introduction of a unary operation
into the signature of the algebra corresponding to arbitrary choices
of inverses for the members of the regular semigroup. Nonetheless,
the bases involved may sometimes be interpreted to provide bases for these e-varieties viewed as 
$\{\mathsf{E}, \mathsf{H}, \mathsf{P}\}$-classes. That approach works well in the cases of $\mathcal{I}$,
$O,$ and $\mathcal{ES},$ the classes of inverse, orthodox, and E-solid
semigroups respectively.  The following is an adaptation of the corresponding
propositions of Hall \cite[Sections~4.2,~4.4, and~4.5]{hal89}. The resulting
bases are quite different from those provided in Theorems \ref{thm:ESolid}, \ref{thm:orthodox},
and \ref{thm:inversesgp}; please note also Remark \ref{rem:reg}.

\begin{theorem}\label{thm:regularlaws} Let $\mathcal{C}$ denote any $\{\mathsf{E},\mathsf{H},\mathsf{P}\}$-class that
admits the equations:
\begin{equation}
(\forall a,b)\,(\exists x,u,v)\colon x\in V(a)\wedge u\in V(a^{2})\wedge v\in V(b^{2}).\tag{\textup{\textsf{reg}}}\label{eq:regularspecial}
\end{equation}
Then any $S\in\mathcal{C}$ is regular and $aua,bvb\in E(S)$. Furthermore
the $\{\mathsf{E},\mathsf{H},\mathsf{P}\}$-classes $\mathcal{I}$, $\mathcal{O}$, and $\mathcal{ES}$
of inverse, orthodox, and E-solid semigroups respectively are defined
by the equation system \eqref{eq:regularspecial} together with the equations\textup:
\begin{equation}
\mathcal{I}\colon \textup{\ref{eq:regularspecial}}\wedge aua\cdot bvb=bvb\cdot aua\textup;\label{eqn:Inv2}
\end{equation}
\begin{equation}
\mathcal{O}\colon \textup{\ref{eq:regularspecial}}\wedge aua\cdot bvb\in E\textup;\label{eqn:O2}
\end{equation}
\begin{equation}
\mathcal{ES}\colon \textup{\ref{eq:regularspecial}}\wedge aua\cdot bvb\in G.\label{eqn:ES2}
\end{equation}
\end{theorem}
\begin{remark}\label{rem:reg} The equation system \eqref{eq:regularspecial} is equivalent to the single
regularity equation $(\forall a)(\exists x)\colon a=axa$ in that this
equation and \eqref{eq:regularspecial} both define the same $\{\mathsf{E},\mathsf{H},\mathsf{P}\}$-class, that being
the class $\mathcal{R}eg$ of all regular semigroups. The redundancies
of \eqref{eq:regularspecial} however serve as a convenient presentational device.
\end{remark}
\begin{proof}[Proof of Theorem \ref{thm:regularlaws}]
 Any semigroup $S$ satisfying \eqref{eq:regularspecial} is obviously regular.
What is more $(aua)^{2}=a(ua^{2}u)a=aua$ is idempotent, as is $bvb$ (where, as is clear, $u,v$ are any witnesses to the existential quantifiers in \eqref{eq:regularspecial}).
If $S$ is inverse, then $S$ satisfies \eqref{eqn:Inv2} as idempotents commute,
while if $S$ is orthodox then $S$ satisfies \eqref{eqn:O2}. If $S$ is E-solid
then the product of idempotents lies in a subgroup (as the core of
$S$ is a union of groups) and so \eqref{eqn:ES2} holds. 

Conversely suppose that $S$ satisfies \eqref{eq:regularspecial} and \eqref{eqn:Inv2} so that $S$ is
regular and take any $a,b\in E(S)$. We then have, 
\[
ab=a^{2}b^{2}=a^{2}ua^{2}\cdot b^{2}vb^{2}=aua\cdot bvb=bvb\cdot aua=b^{2}vb^{2}\cdot a^{2}ua^{2}=b^{2}a^{2}=ba
\]
and so idempotents commute in $S$, which is therefore an inverse
semigroup.

Next assume that $S$ is a (regular) semigroup satisfying \eqref{eq:regularspecial} and
\eqref{eqn:O2} and once again take $a,b\in E(S)$. Then
\[
ab=a^{2}b^{2}=a^{2}ua^{2}\cdot b^{2}vb^{2}=aua\cdot bvb\in E
\]
and so $S$ is orthodox. 

Finally assume that $S$  is a (regular) semigroup satisfying \eqref{eq:regularspecial} and
\eqref{eqn:ES2}. Let $a,e,b\in E(S)$ be such that $a\LL e\RR b.$ Then
\begin{equation}
ab=a^{2}b^{2}=a^{2}ua^{2}\cdot b^{2}vb^{2}=aua\cdot bvb\in G.\label{eqn:ES3}
\end{equation}
On the other hand $ab\in R_{a}\cap L_{b}$ and this together with
\eqref{eqn:ES3} shows that $R_{a}\cap L_{b}$ contains an idempotent $f$ say,
whence $a\RR f\LL b$, thus proving that $S$ is indeed E-solid.
\end{proof}

For any property $T$, a regular semigroup $S$ is called \emph{locally
}$T$ if for every idempotent $e\in E(S)$ the \emph{local subsemigroup
}$eSe$ has property $T$. A regular semigroup $S$ is locally regular
as for any $a\in eSe$, if $x\in V(a)$ then $exe\in V(a)\cap eSe$,
as can be readily checked. We shall have need of the following fact.

\begin{lemma}\label{lem:Disom} \textup(\cite{hal73} and see \cite[Ex.~1.4.11]{hig}.\textup) Let $S$ be
any semigroup and suppose that $e,f\in E(S)$ with $e\DD f$.
Then the subsemigroups $eSe$ and $fSf$ of $S$ are isomorphic. 
\end{lemma}
For a class of regular semigroups $\mathcal{C}$, let $\mathcal{C}^{loc}$
denote the class of all regular semigroups $S$, the local subsemigroups
of which lie in $\mathcal{C}.$ Theorem~\ref{thm:local} below shows that if $\mathcal{C}$
is an $\{\mathsf{E},\mathsf{H},\mathsf{P}\}$-class of regular semigroups then the corresponding
localised class, $\mathcal{C}^{loc}$ is likewise an $\{\mathsf{E},\mathsf{H},\mathsf{P}\}$-class
consisting of regular semigroups and an equational basis for the latter
may be derived, in a systematic fashion, from any basis of the former.
This corresponds to Lemma~4.6.1 of~\cite{hal89}. 

In general, the class $\mathcal{C}'$ of all semigroups $S$ (whether
regular or not) such that $eSe\in\mathcal{C}$ $(e\in E(S))$ is not
closed under $\mathsf{H}$ for by default every free semigroup lies in $\mathcal{C}'$
and so $\mathsf{H}(\mathcal{C}')$ is the class $\mathcal{S}$ of all semigroups.
Hence $\mathcal{C}'$ is closed under $\mathsf{H}$ if and only if for all
semigroups $S$ and $e\in E(S)$ we have $eSe\in\mathcal{C}$. This implies
$\mathcal{C}\supseteq\mathcal{M}$, the class of all monoids. Since $eSe$
is a monoid with identity element $e$, the converse is also true.
It follows that $\mathcal{C}'$ is an $\{\mathsf{E},\mathsf{H},\mathsf{P}\}$-class if and only
if $\mathcal{M}\subseteq\mathcal{C}$, in which case $\mathcal{C}'=\mathcal{S}$.

We now consider a typical $\{\mathsf{E},\mathsf{H},\mathsf{P}\}$-class $\mathcal{C}$ consisting
of regular semigroups. Let $\mathcal{\mathcal{E}=E}(\mathcal{C})$ denote a
set of equation systems that defines $\mathcal{C}$.  Let $\mathcal{E}^{loc}$ be
the set of equation systems derived from $\mathcal{E}$ as follows.  Fix symbols $A$
and $X$ that do not occur in the sentences of $\mathcal{E}$.  For any letter $p$ appearing in $\mathcal{E}$, let $p'$ denote $AXpAX$, and extend this to words $w=p_1\dots p_k$ in the alphabet of $\mathcal{E}$ by letting $w'$ denote $p_1'\dots p_k'$, and then to equation systems $\varepsilon$ by applying $'$ to the words in each equality within $\varepsilon$.  An equation system 
\(
\varepsilon\colon (Q_1 \dots) \dots (Q_k \dots)\colon (u_1 = v_1)\wedge \dots \wedge (u_r = v_r)
\)
becomes
\begin{equation}
\varepsilon'\colon (\forall A)\,(\exists X\in V(A))\,(Q_1 \dots) \dots (Q_k \dots)\colon (u_1' = v_1')\wedge \dots \wedge (u_r' = v_r').\label{eqn:localreplace}
\end{equation}
To construct $\mathcal{E}^{loc}$ from $\mathcal{E}$, we replace each equation system $\varepsilon\in\mathcal{\mathcal{E}}$ by the equation system
$\varepsilon'$.

\begin{theorem}\label{thm:local} Let $\mathcal{C=\mathcal{C}}(\mathcal{E})$ be an $\{\mathsf{E},\mathsf{H},\mathsf{P}\}$-class
consisting of regular semigroups. Then $\mathcal{C}^{loc}=\mathcal{C}(\mathcal{E}^{loc})$.
In particular, $\mathcal{C}^{loc}$ is also an $\{\mathsf{E},\mathsf{H},\mathsf{P}\}$-class of regular semigroups.
\end{theorem}
\begin{proof} 
Consider an equation system $\varepsilon$ from $\mathcal{E}$:
\[
(Q_1 \dots) \dots (Q_k \dots)\colon (u_1 = v_1)\wedge \dots \wedge (u_r = v_r).
\]
So $\varepsilon'$ is 
\[
(\forall A)\,(\exists X)\colon (X\in V(A))\wedge (Q_1 \dots) \dots (Q_k \dots)\colon (u_1' = v_1')\wedge \dots \wedge (u_r' = v_r').
\]
Suppose that $S\in\mathcal{C}^{loc}$. Let $A\in S$ and $X\in V(A)$ be arbitrary, and set $e=AX\in E(S)$.  As $eSe$ satisfies $\varepsilon$, and as each $p\in eSe$ satisfies $p=epe=AXpAX$, it follows immediately that $S$ satisfies $\varepsilon'$.  

Conversely, suppose that $S\in \mathcal{C}(\mathcal{E}^{loc})$ and consider any $\varepsilon\in\mathcal{E}$; we are required to show that $eSe$ satisfies $\varepsilon$, for any choice $e\in E(S)$.  Let $e\in E(S)$ be arbitrary, and note that $S$ satisfies $\varepsilon'$.  Take $A=e$ and let $X\in V(A)$ be the corresponding element of $S$ that exists due to satisfaction of $\varepsilon'$.  Now $f=AX$ is idempotent and each element $p\in fSf$ satisfies $p=fpf=AXpAX$ so that $fSf$ also satisfies $\varepsilon$.  But $e\RR f$, so $e\DD f$, giving $eSe\cong fSf$ by Lemma~\ref{lem:Disom}, showing that $eSe$ satisfies $\varepsilon$, as required.
\end{proof}

\section{Universally Satisfied Equations}\label{sec:universal}

Given an $\{\mathsf{E},\mathsf{H},\mathsf{P}\}$-class $\mathcal{C}$ there are two natural tasks arising.
The first is the determination of an equational basis for $\mathcal{C}$,
which was the subject of Sections 4 and 5. In this section we examine the
other side of the coin, which is the question of finding all equations
satisfied by $\mathcal{C}$. Here we shall solve the latter problem for
one class only, that being the class generated by $P=(\mathbb{Z}^{+},+)$
and for equations of the type $\forall\dots\exists$. As a corollary
we obtain a description of the class of equations without parameters
solvable in every semigroup. 

We shall denote a typical semigroup equation as $e:p=q$ where $p,q\in F_{A\cup X}$,
the free semigroup on $A\cup X$, where $A$ and $X$ are disjoint
countably infinite sets. Elements of $A$ will follow instances of
the $\forall$ quantifier while those drawn from $X$ will follow
the $\exists$ symbol. We shall denote the number of instances of
the letter $y\in A\cup X$ in a word $w\in F_{A\cup X}$ by $|w|_{y}$,
with the length of $w$ simply denoted by $|w|$. Define the \emph{content
}of $w$ as the set 
\[
c(w)=\{y\in X:|w|_{y}\geq1\}.
\]

\begin{definition}  An equation $e\colon p=q$ $(p,q\in F_{A\cup X})$
is \emph{semigroup universal} if $e$ is solvable in every semigroup $S$. 
\end{definition}
In this section we adopt the abbreviation that an equation is \emph{universal} if it is a semigroup-universal equation.  This is is not to be confused with ``universally quantified equation'', which in this article is referred to as an ``identity''.

We will make use of elementary results on subsemigroups of $P$. Our
source here is Chapter 2, Section 4 of the book by Grillet \cite{gri}
who therein gives the original sources of these and other related
facts on numerical semigroups.

\begin{theorem}[Proposition II.4.1 and  Corollary 4.2 of \cite{gri}]\label{thm:gri}

\begin{enumerate}

\item[(i)] Let $S$ be a subsemigroup of $P=(\mathbb{\mathbb{Z}}^{+},+)$.
Then there exists a unique integer $d\geq1$ such that $S$ consists
of multiples of $d$ and $S$ contains all sufficiently large multiples
of $d$. 

\item[(ii)] A subsemigroup $S$ of $(\mathbb{Z},+)$ either contains only
non-negative integers, or only non-positive integers, or is a subgroup
of $(\mathbb{Z},+)$. In the latter case, $S=d\mathbb{Z}$ for some
$d\geq0$. 
\end{enumerate}
\end{theorem}
W shall adopt the convention that gcd$(0,0,\cdots, 0)=0$.
\begin{corollary} \label{cor:gri}
Let $m_{1},\dots,m_{n}\in\mathbb{Z}$ with
gcd$(m_{1},\dots,m_{n})=d$. Define
\[
S=S(m_{1},\dots,m_{n})=\{t_{1}m_{1}+\dots+t_{m}m_{n}:t_{i}\geq1\ (1\leq i\leq n)\}.
\]
 Then $S$ is a subsemigroup of $(\mathbb{Z},+)$. Moreover if all
the integers $m_{i}$ are positive then $S$ is a subsemigroup of
$P=(\mathbb{Z}^{+},+)$ and further $d$ is the unique integer such
that both the conditions that $S\subseteq d\mathbb{Z}^{+}$ and there exists $k\in\mathbb{Z}^{+}$
such that $\{da:a\geq k\}\subseteq S$ are satisfied.
\end{corollary}
\begin{proof} Clearly $S$ is a subsemigroup of $(\mathbb{Z},+)$. In the case
where all the $m_{i}\geq1$, clearly $S \subseteq P$ and by Theorem \ref{thm:gri}(i), there exists a unique
positive integer $d_{1}$ that has both the properties that $S\subseteq d_{1}\mathbb{Z}^{+}$
and there exists a positive integer $k$ such that $d_{1}a\in S$
for all $a\geq k$ . On the other hand $d|s$ for all $s\in S$ so
that $S\subseteq d\mathbb{Z}^{+}$. It follows that $d|kd_{1}$ and
$d|(k+1)d_{1}$, whence $d$ is a factor of their difference, and
so $d|d_{1}$. On the other hand, for any $1\leq i\leq n$ we may
write $m_{1}+\dots+m_{n}=pd_{1}$ and $m_{1}+\dots+2m_{i}+\dots +m_{n}=qd_{1}$
for some $p<q\in\mathbb{Z}^{+}$. Then we have $m_{i}=(q-p)d_{1}$
and so $d_{1}|m_{i}$ for all $1\leq i\leq n$. Therefore $d_{1}|d=\text{gcd\ensuremath{(m_{1},\dots,m_{n}})}$.
We conclude that $d_{1}=d$. 
\end{proof}

\begin{definition}\label{def:gcd} Let $e$ be of the form $(\forall a_{1},\dots,a_{k})\,(\exists x_{1},\dots,x_{l}):p=q$.
Let $r_{i}=|p|_{x_{i}}$, $s_{i}=|q|_{x_{i}},\,p_{j}=|p|_{a_{j}},q_{j}=|q|_{a_{j}}$.
We shall write $r_{i}-s_{i}$ as $m_{i}$ and $q_{j}-p_{j}$ as $n_{j}$;
let $d$ stand for gcd$(m_{1},\dots,m_{l})$ and $d'$ for gcd$(n_{1},\dots,n_{k})$. 
\end{definition}
\begin{theorem}\label{thm:gcd}
 Let $e\colon p=q$ be an equation written in the notation
of Definition \ref{def:gcd}. Then $e\colon p=q$ is solvable in $P$ if and only if for any
given positive integers $a_{i}$ $(1\leq i\leq k)$ there exist positive
integers $t_{i}$ $(1\leq i\leq l)$ \textup(depending on the $a_{i}$\textup) such
that:
\[
\sum_{i=1}^{l}t_{i}m_{i}=\sum_{i=i}^{k}a_{i}n_{i},
\]
which is equivalent to the statement that $S(n_{1},\dots,n_{k})\subseteq S(m_{1},\dots,m_{l})$. 
\end{theorem}
\begin{proof} In $P$, after a substitution $x_{i}\rightarrow t_{i}$,
our equation $p=q$ takes on the form:
\[
(p_{1}a_{1}+\dots+p_{k}a_{k})+(r_{1}t_{1}+\dots+r_{l}t_{l})=(q_{1}a_{1}+\dots+q_{k}a_{k})+(s_{1}t_{1}+\dots+s_{l}t_{l})
\]
\[
\Leftrightarrow\sum_{i=1}^{k}(p_{i}-q_{i})a_{i}+\sum_{i=1}^{l}(r_{i}-s_{i})t_{i}=0
\]
\[
\Leftrightarrow\sum_{i=1}^{l}t_{i}m_{i}=\sum_{i=i}^{k}a_{i}n_{i}.
\]
The integers $r_{i}-s_{i}=m_{i}$ and $q_{i}-p_{i}=n_{i}$, which
are fixed and may be negative, are determined by the equation $e$.
It follows that $e:p=q$ will be solvable in $P$ if and only if every
linear combination of the $n_{i}$ in positive integers is a linear
combination in positive integers of the $m_{i}$, which is the statement of the theorem. 
\end{proof}

\begin{theorem}\label{thm:solvableinP} Let $e\colon p=q$ be an equation written in the notation
of Definition \ref{def:gcd}. Let $d=$ gcd$(m_{1},\dots,m_{l})$ and $d'=$
gcd$(n_{1},\dots,n_{k})$. Suppose for some $x,y\in c(pq)$, $|p|_{x}<|q|_{x}$
and $|p|_{y}>|q|_{y}$. Then $e\colon p=q$ is solvable in $P$ if and only
if $d|d'$. 
\end{theorem}
\begin{proof} By hypothesis, some of the integers $m_{i}$ are positive
and some are negative, from which it follows from Theorem \ref{thm:gri}(ii)
that $S_{1}:=S(m_{1},\dots,m_{l})=d_{1}\text{\ensuremath{\mathbb{Z}} for some \ensuremath{d_{1}\geq1.} 
Since \ensuremath{d_{1}\in S_{1}} it follows that \ensuremath{d|d_{1}.}}$
On the other hand, by the argument in the proof of Corollary \ref{cor:gri},
$d_{1}|d$ also and therefore $d_{1}=d$ and so $S_{1}=d\mathbb{Z}$. 

For $d'\neq0$, $d'|n_{i}$ for all $1\leq i\leq k$ and
it follows that $S(n_{1},\dots,n_{k})\subseteq d'\mathbb{Z}$.
On the other hand $d'=0$ if and only if $n_i = 0$  for all $1\leq i \leq k$ in which case $S(n_1, \cdots, n_k)={0}= d'\mathbb{Z}$.   

Now suppose that $d|d'$ so that $d'=dy$ say. Then 
\[
S(n_{1},\dots,n_{k})\subseteq d'\mathbb{Z}=dy\mathbb{Z}\subseteq d\mathbb{Z}=S(m_{1},\dots,m_{l}),
\]
and so by Theorem \ref{thm:gcd}, $e$ is solvable in $P$.

 Conversely suppose
that $e$ may be solved in $P$. By the argument of the first paragraph of this proof and that of Corollary~\ref{cor:gri}, it follows that $S(n_{1},\dots,n_{k})$
contains a set of the form $\pm \{ad':a\geq k\}$ for some fixed positive
integer $k$. Take $p$ to be a prime with $p\geq k+d$. Then $pd'\in \pm S(n_{1},\dots,n_{k})$.
By Theorem \ref{thm:gcd}, $pd'\in S(m_{1},\dots,m_{l})=d\mathbb{Z}$. We then
have $d|pd'$. However $d$ and $p$ are relatively prime (as $p>d$)
and so $d|d'$, thus completing the proof.
\end{proof} 

\begin{corollary}  Let $e\colon p=q$ be an equation without parameters.
Let $c(pq)=\{x_{1},\dots,x_{n}\}$. Then $e$ is universal if and
only if either\up:
\begin{enumerate}
\item[(i)] $|p|_{x_{i}}=|q|_{x_{i}}$ for all $i=1,2,\dots,n$ or 
\item[(ii)] for some $x,y\in c(pq)$, $|p|_{x}<|q|_{x}$ and $|p|_{y}>|q|_{y}$. 
\end{enumerate}
\end{corollary}
\begin{proof} If (i) applies to $e$ then for any semigroup $S$
take $s\in S$ and substitute $x_{i}\rightarrow s$ $(1\leq i\leq n)$
in $e$; this yields $s^{|p|}=s^{|q|}$, which is true in $S$ as
$|p|=|q|$ and therefore $e$ is universal. Next suppose that (ii)
holds. Then $e$ is equivalent in $P$ to the equation $e':pa=qa$
where $a\in A$ is a parameter. Since $d'=0$, the condition $d|d'$
holds whence it follows from Theorem \ref{thm:solvableinP} that $e'$ is solvable in
$P$ and therefore $e$ is likewise. Now taking any $s\in S$ we have
$\langle s\rangle$ is a homomorphic image of $P$ and hence $e$
is solvable in $\langle s\rangle$. Since $e$ has no parameters,
any solution of $e$ in $\langle s\rangle$ is also a solution of
$e$ in $S$ and therefore $e$ is universal. 

Conversely suppose that neither conditions (i) nor (ii) hold for $e$.
Without loss we may assume that $|p|\leq|q|$. Suppose that for some
$i$ $(1\leq i\leq n)$ $|p|_{x_{i}}>|q|_{x_{i}}$. Then, since $|p|\leq|q|$
it would follow that for some other subscript $j$ we would find that
$|p|_{x_{j}}<|q|_{x_{j}}$, contradicting the assumption that condition
(ii) does not hold. Therefore $|p|_{x_{i}}\leq|q|_{x_{i}}$ for all
$1\leq i\leq n$. Moreover, since condition (i) does not hold either,
for at least one subscript $i$ the previous inequality is strict.
Any substitution $x_{i}\rightarrow t_{i}$ $(t_{i}\in\mathbb{Z}^{+})$
in $P$ therefore yields respective positive integers $p'$ and $q'$
say with $p'<q'$. In particular $p'\neq q'$ and so that $e$ cannot
be satisfied in $P$ and therefore $e$ is not universal. 
\end{proof}

\begin{example} Let us consider 
\[
e\colon x_{1}^{9}x_{2}^{23}a_{1}^{2}a_{2}^{13}a_{3}=x_{1}^{30}x_{2}^{8}a_{1}^{11}a_{2}^{7}a_{3}^{10}.
\]
In additive notation, which is more fitting for $P$,  our equation $e$ takes the form:
\[
e\colon  9x_1 + 23x_2 + 2a_1 + 13a_2 + a_3 = 30x_1 + 8x_2 + 11a_1 + 7a_2 +10a_3.
\]
Here $m_{1}=9-30=-21$, $m_{2}=23-8=15$, and $d=$ gcd($m_{1},m_{2})=3$;
$n_{1}=11-2=9,\,n_{2}=7-13=-6,\,n_{3}=10-1=9,$ and $d'=$ gcd$(9,-6,9)=3$.
Then $d=d'$ so $d|d'$ and by Theorem \ref{thm:solvableinP}, $e$ is solvable in $P$. Particular
selections for $a_{1},a_{2}$, and $a_{3}$ lead to solvable linear
diophantine equations in the respective positive integer multipliers $t_1$ and $t_2$ of  $m_1$ and $m_2$. 
\end{example}

\begin{example} The equation (again written additively):
\[
e\colon  13x + 24y + 2a + 5b = 10x + 16y + 13a + 19b
\]
is an instance in which each of the variables $(x$ and $y$) occurs
more often on the left than they do on the right so that Theorem \ref{thm:solvableinP}
does not apply. Nonetheless Theorem \ref{thm:gcd} shows us that $e$ is solvable
in $P$. We have $m_{1}=13-10=3,$$\,m_{2}=24-16=8$, $n_{1}=13-2=11,\,n_{2}=19-5=14$.
We note that $34=(6\times3)+(2\times8),\,35=(1\times3)+(4\times8)$,
$36=(4\times3)+(3\times8)$. It follows that 
\[
\{k\in\mathbb{Z}:k\geq34\}\subseteq\{3t_{1}+8t_{2}:t_{1},t_{2}\geq1\}= S(m_1, m_2).
\]
Now $\{11t_{1}+14t_{2}:t_{1},t_{2}\geq1\}\cap\{k\in\mathbb{\mathbb{Z}}:k\leq33\}=\{11+14=25\}$
and $(3\times3)+(2\times8)=25\in S(m_1, m_2)$ also. 
Hence $S(n_1, n_2) \subseteq S(m_1, m_2)$
and so $e$ is solvable in $P$. 

For a particular instance we put
$a=2$ and $b=3$ giving the diophantine equation for the multipliers $t_1$ and $t_2$:
\[
3t_{1}+8t_{2}=(2\times11)+(3\times14)=64;
\]
hence $2t_{2}\equiv1$ (mod $3)$ so $t_{2}=2+3t$, giving $3t_{1}+8(2+3t)=64$,
whence $3t_{1}=48-24t$ and so $t_{1}=16-8t$. Since $t_{1},t_{2}\geq1$,
there are two solutions given by $t=0,1$ which are respectively $t_{1}=16,$
$t_{2}=2$ and $t_{1}=8$, $t_{2}=5$. Substituting $x=16$ and $y=2$
yields a common value of $275$ for both sides of the equation $e$,
while putting $x=8$ and $y=5$ gives $243$ on each side.
\end{example}

\bf  Acknowledgement\normalfont: The authors would like to express their gratitude to the referee for their accurate and detailed comments and criticisms of the original draft of this paper.

\end{document}